\documentclass[11pt]{amsart} \textwidth=14.5cm \oddsidemargin=1cm
\usepackage[rgb]{xcolor}
\usepackage{tikz}
\usepackage{verbatim}
\usepackage{amsmath}
\usepackage{amsxtra}
\usepackage{amscd}
\usepackage{amsthm}
\usepackage{amsfonts}
\usepackage{amssymb}
\usepackage{eucal}
\usepackage[new]{old-arrows}
\allowdisplaybreaks

\usepackage[linktocpage=true]{hyperref}
\usepackage[all, cmtip]{xy}
\usepackage{stmaryrd}
\usepackage{color} 

\newtheorem{thm}{Theorem} [section]
\newtheorem{cor}[thm]{Corollary}
\newtheorem{lem}[thm]{Lemma}
\newtheorem{prop}[thm]{Proposition}

\theoremstyle{definition}
\newtheorem{definition}[thm]{Definition}

\theoremstyle{remark}
\newtheorem{rem}[thm]{Remark}

\numberwithin{equation}{section}

\begin{document}

\newcommand{\thmref}[1]{Theorem~\ref{#1}}
\newcommand{\secref}[1]{Section~\ref{#1}}
\newcommand{\lemref}[1]{Lemma~\ref{#1}}
\newcommand{\propref}[1]{Proposition~\ref{#1}}
\newcommand{\corref}[1]{Corollary~\ref{#1}}
\newcommand{\remref}[1]{Remark~\ref{#1}}
\newcommand{\eqnref}[1]{(\ref{#1})}

\newcommand{\exref}[1]{Example~\ref{#1}}

\newcommand{\nc}{\newcommand}
 \nc{\Z}{{\mathbb Z}}
 \nc{\C}{{\mathbb C}}
 \nc{\N}{{\mathbb N}}
 \nc{\F}{{\mf F}}
 \nc{\Q}{\mathbb{Q}}
 \nc{\la}{\lambda}
 \nc{\ep}{\epsilon}
 \nc{\h}{\mathfrak h}
 \nc{\n}{\mf n}
 \nc{\G}{{\mathfrak g}}
 \nc{\DG}{\widetilde{\mathfrak g}}
 \nc{\SG}{\breve{\mathfrak g}}
 \nc{\is}{{\mathbf i}}
 \nc{\V}{\mf V}
 \nc{\bi}{\bibitem}
 \nc{\E}{\mc E}
 \nc{\ba}{\tilde{\pa}}
 \nc{\half}{\frac{1}{2}}
 \nc{\hgt}{\text{ht}}
 \nc{\mc}{\mathcal}
 \nc{\mf}{\mathfrak} 
 \nc{\hf}{\frac{1}{2}}
\nc{\ov}{\overline}
\nc{\ul}{\underline}
\nc{\I}{\mathbb{I}}
\nc{\aaa}{{\mf A}}
\nc{\xx}{{\mf x}}
\nc{\id}{\text{id}}
\nc{\one}{\bold{1}}
\nc{\Qq}{\Q(q)}
\nc{\mA}{\mathcal{A}}
\nc{\mH}{\mathcal{H}}
\nc{\bk}{\mathbb{K}}

\newcommand{\U}{\bold{U}}
\newcommand{\Ui}{\bold{U}^\imath}
\newcommand{\VV}{\mathbb{V}}
\newcommand{\B}{\bold{B}}

\newcommand{\blue}[1]{{\color{blue}#1}}
\newcommand{\red}[1]{{\color{red}#1}}
\newcommand{\green}[1]{{\color{green}#1}}
\newcommand{\white}[1]{{\color{white}#1}}

\newcommand{\htodo}{\todo[inline,color=orange!20, caption={}]}

\title[Positivity of $\imath$-Canonical bases]
{Positivity of $\imath$-Canonical bases}
 
 \author[Huanchen Bao]{Huanchen Bao}
\address{Department of Mathematics, University of Maryland, College Park, MD 20742}
\email{huanchen@math.umd.edu}

\begin{abstract}
For the quantum symmetric pair $(\U,\Ui)$ of type AIII/AIV, we show various positivity properties of the $\imath$-canonical bases on finite-dimensional simple $\U$-modules, as well as their tensor product.

\end{abstract}

\maketitle

\let\thefootnote\relax\footnotetext{{\em 2010 Mathematics Subject Classification.} Primary 17B10.}




\section*{Introduction}
We write $\U = \U_q(\mathfrak{sl}_{n+1})$ as the quantum group over the field $\Q(q)$ with a generic parameter $q$ through out this paper. We denote by $\N$ the set of non-negative integers.

Let $(\U, \Ui)$ be the quasi-split (that is, without black dots in the Satake diagram) quantum symmetric pair of type AIII/AIV. The structure theory of quantum symmetric pairs was established by Letzter (\cite{Le02}). 

It has been recently discovered that this quantum symmetric pair is closely related with the BGG category $\mc{O}$ of type $B/C/D$ (\cite{BW13, ES13}). The canonical bases arising from this quantum symmetric pair have been constructed in \cite{BW13}, which were then used to reformulate the Kazhdan-Lusztig theory of type B/C. This construction has been adapted in \cite{Bao17} for the type D setting. The related geometric construction and categorification have been obtained in \cite{BKLW} and \cite{BSWW}, respectively. 

We denote by $X$ the integral weight lattice of $\U$, and denote by $X^+$ the set of dominant weights. For any $\lambda_1, \dots, \lambda_k \in X^+$, the tensor product of finite-dimensional simple $\U$-modules $L(\lambda_1) \otimes \cdots \otimes L(\lambda_k)$ admits the canonical basis $B(\lambda_1, \dots, \lambda_k)$ (\cite{Lu94}), as well as the $\imath$-canonical basis $B^\imath(\lambda_1, \dots, \lambda_k)$ (\cite{BW13, Bao17}). 
We prove the following positivity results in this paper.

\begin{thm}[Theorems~\ref{thm:BW13} and \ref{thm:Bao17}]Let $ 0 \le l \le k$.
We have 
\[
b^\imath = \sum_{b^\imath_1, b_2} t_{b; b_1, b_2} b^\imath_{1} \otimes b_2, \qquad \text{with } t_{b; b_1, b_2} \in \N[q],\]
where $b^\imath \in B^\imath(\lambda_1, \dots, \lambda_k)$, $b^\imath_1 \in B^\imath(\lambda_1 , \dots, \lambda_l)$, and $b_2 \in B(\lambda_{l +1}, \dots, \lambda_k )$.
\end{thm}

We actually identify these coefficients $t_{b; b_1, b_2}$ with certain polynomials arising from the (variations of) Kazhdan-Lusztig bases of various Hecke algebras. In  extreme cases, they are exactly the (parabolic) Kazhdan-Lusztig polynomials. To establish such identification, we use Lusztig's theory of based modules (\cite[Chap. 27]{Lu94}), as well as its $\imath$-counterpart for quantum symmetric pairs. 

Special case of the theorem has already been obtained in \cite{FL15}, where they consider the case all $L(\lambda_i)$ being the natural representation $\VV$ and $l=0$.

We have the following positivity result on finite-dimensional simple $\U$-modules, which answers affirmatively \cite[Conjecture~4.24]{BW13}. 
\begin{cor}
Let $\lambda \in X^+$. Let $B^\imath (\lambda)$ and $B(\lambda)$ be the sets of $\imath$-canonical basis  and canonical basis on $L(\lambda)$, respectively. We have 
\[
b^\imath = \sum_{b_1 \in B} t_{{b; b_1}} b_1, \qquad \text{ with }  t_{{b; b_1}} \in \N[q].
\]
\end{cor}

Our method can easily be adapted to give a simple proof (see \cite{We15} for a proof via categorification) of a similar positivity statement for Lusztig's canonical bases, that is, for $b \in B(\lambda_1, \dots, \lambda_k)$, $b_1 \in B(\lambda_1 , \dots, \lambda_l)$, and $b_2 \in B(\lambda_{l +1}, \dots, \lambda_k )$, we have
\[
b = \sum_{b^\imath_1, b_2} t'_{b; b_1, b_2} b_{1} \otimes b_2, \qquad \text{with } t'_{b; b_1, b_2} \in \N[q].
\]


{\bf Acknowledgement:} The paper was inspired by conversations with Yiqiang Li and Weiqiang Wang during the author's visit to UVa. We would like to thank them for the discussion. Early results were presented in the algebra seminar of UVa. We would like to thank Weiqiang Wang for the invitation and continued encouragement throughout the years.



\section{The Hecke algebras} \label{sec:Hecke}
\subsection{}
Let $(W, S)$ be a finite Weyl group. We denote the length function by $\ell ( \cdot)$. For any $J \subset S$, we denote by $W_J$ the corresponding parabolic subgroup, and denote by $w_J$ the longest element of $W_J$. 

Let $W^J$ and ${}^JW$ be the set of minimal length coset representatives of $W/W_J$ and $W_J\backslash W$, respectively. 
For $I, J \subset S$, we denote by ${}^I W^J$ the set of minimal length double coset representatives of $W_I \backslash W/ W_J$. For any (left, right, or double) coset $p$, we denote by $p_-$ the minimal length representative of this coset. We shall abuse the notation and denote the Bruhat orders on $W$, as well as on any cosets, by $\le$. The following lemma can be found in \cite[Theorems~2.7.4 and 2.7.5]{Ca93}.
\begin{lem}\label{lem:HK}
Let $I, J \subset S$ and $p \in W_I \backslash W/ W_J$. Let $K = I \cap p_- J p^{-1}_-$.
	\begin{enumerate}
	\item We have 
\[
W_I \cap p_- W_J p^{-1}_- = W_{I \cap p_- J p^{-1}_-}.
\]
	\item The map 
\[
	\begin{split}
(W^K \cap W_I) \times W_J &\longrightarrow p\\
	(u, v) &\mapsto up_-v
	\end{split}
\]
is a bijection and $\ell (up_-v) = \ell (u) + \ell( p_-) + \ell( v)$.
\end{enumerate}
\end{lem}

\subsection{} 

Let $\mc{H} = \mc{H}_W$ be the Hecke algebra of $(W,S)$.
This is a $\Qq$-algebra with the standard basis $\{H_w \vert w \in W\}$ satisfying the relations:
\[
\begin{split}
H_v H_w = H_{vw}, \qquad &\text{ if } \ell(vw) = \ell(v) + \ell(w);\\
(H_s +q)(H_s-q^{-1})=0, \qquad &\text{ for } s \in S.
\end{split}
\] 
Let $\bar{\,} : \mH \rightarrow \mH$ be the $\Qq$-semilinear bar involution such that $\overline{H_s}= H_s^{-1}$ and $\overline{q} =q^{-1}$. 

Thanks to \cite[Chap~5]{Lu03}, for any $w \in W$, there is a unique element $\underline{H_w}$ such that
\begin{enumerate}
	\item	$\overline{\underline{H_w}} =\underline{H_w}$;
	\item	$\underline{H_w} = \sum_{y \in W} p_{y,w} H_y$ where
		\begin{itemize}	
			\item	$p_{y,w} = 0$ unless $y \le w$;
			\item	$p_{w,w} =1$;
			\item	$p_{y,w} \in q \Z[q]$ if $y < w$.
		\end{itemize}
\end{enumerate}
The set $\{\underline{H_w} \vert w \in W\}$ forms a $\Qq$-basis of $\mH$. This is the canonical (or Kazhdan-Lusztig) basis of $\mc{H}$.


\subsection{}\label{sub:MJ} 
We then define the parabolic canonical basis, which was originally defined in \cite{Deo87}. Modules of Hecke algebras are usually right modules.

Let $J \subset S$. Let $e_J^+ (H_w) = q^{-\ell(w)}$ be the $1$-dimensional trivial representation of $\mc{H}_J$. 
We define the induced module $M_J  = \text{Ind}_{\text{-}\mc{H}_{J}}^{\text{-}\mc{H}} e^+_J = M_e \cdot \mc{H}$, where the generator $M_e$ satisfies $M_e \cdot H_{w} = q^{-\ell(w)} M_e$ for $w \in W_J$. The module $M_J$ admits a standard basis $\{M_w = M_e \cdot H_w \vert w \in {}^JW\}$. This module admits an $\Qq$-semilinear bar involution compatible with the bar involution on $\mc{H}$, such that
\[
 \overline{M_e} = M_e, \qquad \overline{m \cdot h} =\overline{m} \cdot  \overline{h}, \quad\text{ for } h \in \mc{H}, m \in M_J.
\] 

Hence similar to \cite[Chap. 5]{Lu03}, for any $w \in {}^JW$, there is a unique element $\underline{M_w}$ such that
\begin{enumerate}
	\item	$\overline{\underline{M_w}} =\underline{M_w}$;
	\item	$\underline{M_w} = \sum_{y \in {}^JW} p^+_{y,w} {M_y}$ where
		\begin{itemize}	
			\item	$p^+_{y,w} = 0$ unless $y \le w$;
			\item	$p^+_{w,w} =1$;
			\item	$p^+_{y,w} \in q \Z[q]$ if $y < w$.
		\end{itemize}
\end{enumerate}
The set $\{\underline{M_w} \vert w \in {}^JW\}$ forms the canonical basis of $M_J$.
We actually have the following embedding of $\mc{H}$-modules
\begin{equation}\label{eq:pJ}
\begin{split}
	 p_J^+: M_J &\rightarrow  \mc{H},    \\
	 M_w &\mapsto \underline{H_{w_J}} \cdot H_w, \quad w \in {}^JW,\\
	 \underline{M_w} &\mapsto \underline{H_{w_Jw}}, \quad w \in {}^J W.
 \end{split}
\end{equation}
In a more concrete way, we have
\begin{equation}\label{eq:p+}
\sum_{y\in {}^JW} p^+_{y;w} \underline{H_{w_J}} \cdot H_y = \underline{H_{w_Jw}}.
\end{equation}



\subsection{}

Let $I \subset S$. We are interested in the restriction of $\mc{H}$ and $M_J$ as (right) $\mc{H}_I$-modules. Let us start with the module $\mc{H}$. We have 
\begin{equation}\label{eq:Hp}
	\mc{H} \cong \bigoplus_{p \in W^I} \, H_{p_-} \cdot \mc{H}_I  , \quad \text{ as } \mc{H}_I\text{-modules}. 
\end{equation}

For any $w = p_- w'$ with $p \in W^I$ and $w' \in W_I$, we write $\underline{H^I_{w}} =  H_{p_-} \cdot \underline{H_{w'}}$. It is easy to see that $\{\underline{H^I_{w}} \vert w \in W\}$ forms a basis of $\mc{H}_{W}$. This is the so-called hybrid basis in \cite{GH}.

%
We then turn to the module $M_J$. We have 
\begin{equation}\label{eq:J+I}
	M_J \cong \bigoplus_{p \in {}^JW^I}M_e \cdot H_{p_-} \cdot \mc{H}_I = \bigoplus_{p \in {}^JW^I}M_{p_-} \cdot \mc{H}_I, \text{ as } \mc{H}_I\text{-modules}.
\end{equation}

Now thanks to Lemma~\ref{lem:HK}, we have 
\[
M_{p_-} \cdot \mc{H}_I \cong \text{ind}_{\text{-}\mc{H}_K}^{\text{-}\mc{H}_I} e^+_{K}, \qquad \text{ where } K = K_p =p^{-1}_- J p_- \cap I.
\]
Therefore we have the standard basis $\{M_{p_-w} \vert w \in W_I \cap {}^{K_p}W \}$, and can also define the canonical basis $\{ \underline{M_{{p_-}w}^I} \vert w \in W_I \cap {}^{K_p}W \}$ of each summand $M_{p_-} \cdot \mc{H}_I$ follwing \S\ref{sub:MJ}. Therefore we obtain the (parabolic) hybrid basis $\{\underline{M_{p_-w}^I} \vert w \in W_I \cap {}^{K_p}W,  K_p =p^{-1}_- J p_- \cap I, p \in {}^J W^I\}$ of $M_J$ as an $\mc{H}_I$-module. Recall the embedding $p^+_J : M_J \rightarrow \mc{H} $ in \eqref{eq:pJ}.

\begin{lem}\label{lem:p+}
We fix a double coset $p \in W_J \backslash W / W_I$. Let $K = p^{-1}_- J p_- \cap I$, $K' = p_- K p^{-1}_- = J \cap p_- I p^{-1}_-$, and $w \in  W_I \cap {}^KW $.  We have
\[
p^+_J( \underline{M_{p_-w}^I}) = \sum_{r \in W_J \cap W^{K'}}q^{\ell(w_J)- \ell(rw_{K'}p_{-} w) + \ell(p_-w) }\underline{H^I_{rw_{K'} p_- w}}.
\]
\end{lem}

\begin{proof}
Recall that we have 
\begin{align*}
\underline{H_{w_J}} = \sum_{w \in W_J} q^{\ell(w_J) - \ell(w)} H_w=  \sum_{r \in W_J \cap W^{K'}} q^{\ell(w_J)- \ell(w_{K'}) - \ell(r) } H_{r} \cdot \underline{H_{w_{K'}}}.
\end{align*}
On the other hand, thanks to Lemma~\ref{lem:HK} we have 
\[
 \underline{H_{w_{K'}}} \cdot H_{p_-} = H_{p_-} \cdot \underline{H_{w_{K}}}.
\]

Hence we have 
\[
\underline{H_{w_J}} \cdot H_{p_-} =  \sum_{r \in W_J \cap W^{K'}} q^{\ell(w_J)- \ell(w_{K}) - \ell(r) } H_{r} \cdot H_{p_-} \cdot \underline{H_{w_{K}}}.
\]

Now we can write 
\begin{align*}
p^+_J(\underline{M_{p_-w}^I}) &= p^+_J( \sum_{y \in W_I \cap {}^KW} p^+_{y;w}M_{p_-} \cdot H_y)\\
&= \sum_{y \in W_I \cap {}^KW} p^+_{y;w} \underline{H_{w_J}} \cdot H_{p_-} \cdot H_y \\
&= \sum_{y \in W_I \cap {}^KW} p^+_{y;w} \Big( \sum_{r \in W_J \cap W^{K'}} q^{\ell(w_J)- \ell(w_{K}) - \ell(r) } H_{r} \cdot H_{p_-} \cdot \underline{H_{w_{K}}} \cdot H_y \Big) \\
&= \sum_{r \in W_J \cap W^{K'}}q^{\ell(w_J)- \ell(w_{K}) - \ell(r) } H_{r} \cdot H_{p_-} \cdot \Big( \sum_{y \in W_I \cap {}^KW} p^+_{y;w}   \cdot \underline{H_{w_{K}}} \cdot H_y \Big)\\
 &\stackrel{\heartsuit}{=}\sum_{r \in W_J \cap W^{K'}}q^{\ell(w_J)- \ell(w_{K}) - \ell(r) } H_{r p_-}  \cdot \underline{H_{w_K w}}\\
&=\sum_{r \in W_J \cap W^{K'}}q^{\ell(w_J)- \ell(w_{K'}) - \ell(r) }\underline{H^I_{rw_{K'} p_- w}}\\
&=\sum_{r \in W_J \cap W^{K'}}q^{\ell(w_J)- \ell(rw_{K'}p_{-} w) + \ell(p_-w) }\underline{H^I_{rw_{K'} p_- w}}.
\end{align*}
The identity $(\heartsuit)$ follows from \eqref{eq:p+}. This finishes the proof. 
\end{proof}


\subsection{} The following proposition is the key for the positivity results in this paper. 
\begin{prop}\cite{GH, Bra03}\label{prop:GH}
We have, as an equation in $\mc{H}$,
\[
	\underline{H_{w}}  = \sum_{y \le w}p^I_{y ; w}  \underline{H^I_{y}}, \quad p^I_{y ; w} \in \N[q].
\]

\end{prop}

\begin{prop}\label{prop:MJ}
Let $ w \in {}^JW$. We have the following positivity result: 
\[
\underline{M_w} = \sum_{y \in {}^JW } p^{I, +}_{y,w} \underline{M^I_y} \in M_J, \qquad  \text{with }p^{I, +}_{y,w} \in \N[q].
\]
\end{prop}

\begin{rem}
The proposition can also be proved similar to Proposition~\ref{prop:GH} via hyperbolic localization on partial flag varieties following \cite{GH, Bra03}. 

We give another proof identifying $p^{I, +}_{y,w}$ with $p^I_{y', w'}$ up to a ``positive" scalar for certain $y'$, $w' \in W$. This should be considered as an analog of \cite[Proposition~3.4]{Deo87}.
\end{rem}

\begin{proof}
We compare the images of both sides under the embedding \eqref{eq:pJ}. 

For the left hand side, we have
\[
p^+_J(\underline{M_{w}}) = \underline{H_{w_Jw}}= \sum_{y \in W} p^I_{y, w_Jw} \underline{H^I_{y}}.
\]

Thanks to Lemma~\ref{lem:p+}, we have 
\begin{align*}
&p^+_{J}(  \sum_{y \in {}^JW } p^{I, +}_{y,w} \underline{M^I_y} ) \\
= &p^+_{J}\Big( \sum_{p \in W_J \backslash W / W_I} \Big( \sum_{y \in {}^JW, y \in p } p^{I, +}_{y,w} \underline{M^I_y} \Big) \Big) \\
=& \sum_{p \in W_J \backslash W / W_I} \Big( \sum_{y \in {}^JW, y \in p } p^{I, +}_{y,w}\Big( \sum_{r \in W_J \cap W^{K_p'}}q^{\ell(w_J)- \ell(rw_{K_p'}y) + \ell(y) }\underline{H^I_{rw_{K_p'} y}} \Big)\Big),
\end{align*}
where $K'_p = J \cap p_- I p^{-1}_-$.

Then by comparing the coefficients, we have 
\[
q^{\ell(rw_{K'_p} y)}p^{I}_{rw_{K'_p} y, w_Jw} = q^{\ell(w_J) + \ell(y)} p^{I, +}_{y,w} , \quad \text{ for } {y \in p \in W_J \backslash W / W_I}.
\]
The proposition follows from the positivity in Proposition~\ref{prop:GH}. 
\end{proof}
\section{Based modules of quantum groups}
In this section, we review based modules of quantum groups. The main reference is \cite[Chap.~27]{Lu94}. We denote the standard generators of $\U = \U_q(\mathfrak{sl}_{n+1})$ by $E_i, F_i, K_i^{\pm 1}$ for $i =1, 2, \dots, n$, and denote the anti-linear ($q \mapsto q^{-1}$) bar involution on $\U$ by $\psi$ (denoted by $\,\bar{\,}\,$ in \cite[\S3.1.12]{Lu94}). 

We use the following convention on the coproduct of $\U$: 
\begin{align*} 
\begin{split}
\Delta:  &\U \longrightarrow \U \otimes \U,
 \\
 \Delta (E_{i}) &= 1 \otimes E_{i} + E_{i} \otimes K^{-1}_{i}, \\
 \Delta (F_{i}) &= F_{i} \otimes 1 +  K_{i} \otimes F_{i},\\
 \Delta (K_{i}) &= K_{i} \otimes K_{i}.
 \end{split}
\end{align*}
This is the same as the one used in \cite{BW13}, but is different from the one used in \cite{Lu94}.

\subsection{}
Let $(M,B)$ be a finite-dimensional based $\U$-module (\cite[\S27.1.2]{Lu94}). We shall denote the associated bar involution on $M$ by $\psi$. We shall abuse the notation and denote by $\psi$ the bar involution on any based $\U$-module. Elements of $B$ are $\psi$-invariant by definition. We have the following compatibility of bar involutions:
\[
\psi (u x) = \psi(u) \psi(x), \qquad u \in \U, x \in M.
\]
Let $(M\rq{},B\rq{})$ be another based $\U$-module. A based morphism from a based module $(M,B)$ to a based module $(M\rq{},B\rq{})$ is by definition a morphism $f : M \rightarrow M\rq{}$ of $\U$-modules such that 
\begin{enumerate}
	\item 	for any $b \in B$, we have $f(b) \in B\rq{} \cup \{0\}$;
	\item 	$B \cap \ker f$ is a basis of $\ker f$.
\end{enumerate}
As a consequence, we see that $f \circ \psi = \psi \circ f$. We also have the obvious notions of based submodules and based quotients.

\begin{lem}\cite[Theorem~27.3.2]{Lu94}\label{lem:1}
The tensor product $M \otimes M\rq{}$ is a based $\U$-module with the basis $B \diamondsuit B\rq{} = \{b \diamondsuit b\rq{} \vert (b, b\rq{}) \in B \times B\rq{}\}$, where $b \diamondsuit b\rq{}$ is the unique $\psi$-invariant element of the form
\begin{equation}\label{eq:1}
b\rq{} \otimes b\rq{} + \sum_{(b_1 , b_1\rq{}) \in B \times B\rq{}} q \Z[q] b_1 \otimes b_1\rq{}.
\end{equation}

Here the bar involution $\psi$ on $M_1 \otimes M_2$ is defined as $\psi= \Theta \circ (\psi \otimes \psi)$, where $\Theta$ denotes the quasi-$\mc{R}$ matrix (\cite[\S4.1]{Lu94}). 
\end{lem}

\begin{rem}
Note that while it is crucial to consider the integral lattice $\Z[q,q^{-1}]$ in the theory of based modules, it plays little role (after the establishment of various canonical bases) for the question we are interested in in this paper. So we shall not make this explicit.
\end{rem}

\begin{lem}\label{lem:tensor}Let $(M_i, B_i)$ be finite-dimensional based $\U$-modules for $i = 1, 2,3,4$. Let $f: M_1 \rightarrow M_3$ and $g: M_2 \rightarrow M_4$ be morphisms of based $\U$-modules. 
Then the morphism $f \otimes g: M_1 \otimes M_2 \rightarrow M_3 \otimes M_4$ is a based morphism of $\U$-modules. 
\end{lem}

\begin{proof} 
It is clear that $f \otimes g$ is a morphism of $\U$-modules. We prove it is based.

Since $f$ and $g$ are both based morphisms, we know that $(f \otimes g) \circ (\psi \otimes \psi) = (\psi \otimes \psi) \circ (f \otimes g)$. Since $\Theta$ is in the completion of $\U\otimes \U$ (\cite[\S4.1]{Lu94}), we see that $f \otimes g$ commutes with $\Theta$ automatically. Hence $f \otimes g$ commutes with the bar involutions $\psi =  \Theta \circ (\psi \otimes \psi)$. 

Let $b_1 \diamondsuit b_2 \in B_1 \diamondsuit B_2$. We have 
\[
\psi\Big((f \otimes g)( b_1 \diamondsuit b_2)\Big) = (f \otimes g) \Big( \psi(b_1 \diamondsuit b_2)\Big) = (f \otimes g)( b_1 \diamondsuit b_2).
\] 

On the other hand, $(f \otimes g)( b_1 \diamondsuit b_2)$ is of the form (thanks to \eqref{eq:1}) 
\[
f(b_1) \otimes g(b_2) + \sum_{(b_1\rq{} , b_2\rq{}) \in B_1 \times B_2} q \Z[q] f(b_1\rq{}) \otimes g(b_2\rq{}),
\]
where $f(b_1), f(b_1\rq{}) \in B_3 \cup \{0\}$ and $g(b_2), g(b_2\rq{}) \in B_4  \cup \{0\}$, since $f$ and $g$ are based.

Now if neither $f(b_1)$ nor $g(b_2)$ are zero, thanks to Lemma~\ref{lem:1}, we see that $(f \otimes g)( b_1 \diamondsuit b_2) = f(b_1) \diamondsuit g(b_2) \in B_3 \diamondsuit B_4$.

 if $f(b_1)$ or $g(b_2)$ is zero, $(f \otimes g)( b_1 \diamondsuit b_2)$ is a $\psi$-invariant elment of the form  
 \[\sum_{(b_1\rq{} , b_2\rq{}) \in B_1 \times B_2} q \Z[q] f(b_1\rq{}) \otimes g(b_2\rq{}),\] hence has to be zero. Therefore we have
\[
(f \otimes g) (b_1 \diamondsuit b_2) =  
\begin{cases}f(b_1) \diamondsuit g(b_2), &\text{if } (f(b_1), g(b_2)) \in B_3 \times B_4;\\
0, &\text{otherwise}.
\end{cases}
\]
It follows immediately that $f \otimes g$ is based.
\end{proof}



\subsection{}\label{sec:A} Let $\I = \{1, 2, \dots, n+1\}$ be the index set.
Let $\VV$ be the natural representation of $\U$. Then $\VV$ is a based $\U$-module with the basis $B = \{v_i \vert i \in \I\}$ being the standard basis as well as the canonical basis. The tensor product $\VV^{\otimes m}$ admits the standard basis $\{v_f = v_{f(1)} \otimes \cdots \otimes v_{f(m)} \vert f \in \I^m\}$. It is also a based module with basis $B^{\diamondsuit m}$ thanks to Lemma~\ref{lem:1}. We call an element $f \in \I^m$ anti-dominant, if $f(1) \le f(2) \le \cdots \le f(m)$. 

Let $1 \le m \le n$. Let $\mc{H}_{A_{m-1}}$ be the Hecke algebra of the Weyl group $W_{A_{m-1}}$, where we denote the set of simple reflections of $W_{A_{m-1}}$ by $S$. The Hecke algebra  $\mc{H}_{A_{m-1}}$ and the Weyl group $W_{A_{m-1}}$ both act (from the right) naturally on $\VV^{\otimes m}$ (c.f. \cite[\S5.1]{BW13}).  We have the following identification of $\mc{H}_{A_{m-1}}$-modules,
\begin{align*}
\VV^{\otimes m} =&  \bigoplus_{f \in \I^m} v_f \cdot \mc{H}_{A_{m-1}} \cong \bigoplus_{ f \in \I^m }\underline{H_{w_{J(f)} }} \cdot \mc{H}_{A_{m-1}}
 \cong   \bigoplus_{ f \in \I^m}  M_{J(f)},  \text{ for anti-dominant } f  \in \I^m,
\end{align*}
where $J(f) \subset S$ such that $W_{J(f)} \subset W_{A_{m-1}}$ is the stablizer of $v_f$.  

\begin{prop}\cite{FKK98}\label{prop:A}
The basis  $B^{\diamondsuit m}$ of $\VV^{\otimes m}$ can be identified with the (parabolic) Kazhdan-Lusztig basis associated with the Hecke algebra $\mc{H}_{A_{m-1}}$ via the  $q$-Schur duality. 
\end{prop}

 We define  $R = \sum_{w \not = e} \VV^{\otimes m} \cdot \underline{H_{w}}$, for $\underline{H_{w}} \in \mc{H}_{A_{m-1}}$.
\begin{lem}\label{lem:wedge} 
The quotient space $\wedge^m \VV = \VV^{\otimes m}/ R$ is a based $\U$-module. Moreover, the quotient map $\pi : \VV^{\otimes m} \rightarrow \wedge^m \VV$ is a morphism of based $\U$-modules.
\end{lem}

\begin{proof}
It suffices to prove that $R$ is a based $\U$-submodule of $\VV^{\otimes m}$. It follows from the $q$-Schur duality that $R$ is a $\U$-submodule. We prove it is based. 

Recall $\VV^{\otimes m} =  \bigoplus_{ \text{anti-dominant }f \in \I^m} v_f \cdot \mc{H}_{A_{m-1}}$. For any anti-dominant $f$ with $J(f) \neq \emptyset $, we have 
\begin{align*}
&\sum_{e \neq w \in W} v_f \cdot  \mc{H}_{A_{m-1}}\cdot \underline{H_{w}} \\
\cong &\sum_{e \neq w \in W}  \underline{H_{w_{J(f)} }} \cdot  \mc{H}_{A_{m-1}} \cdot \underline{H_w} \\ 
=& \sum_{e \neq w \in W} \underline{H_{w_{J(f)} }} \cdot  \mc{H}_{A_{m-1}} \\
\cong\, &v_f\cdot  \mc{H}_{A_{m-1}}.
\end{align*} 

So in this case,  the subspace $\sum_{e \neq w \in W} v_f \cdot  \mc{H}_{A_{m-1}}\cdot \underline{H_{w}}$ admits basis $\{v_f \underline{{H}_{w}}\vert w \in {}^{J(f)}W \}$. On the other hand, if $f$ is anti-dominant with $J(f) = \emptyset$, the subspace $\sum_{e \neq w \in W} v_f \cdot  \mc{H}_{A_{m-1}}\cdot \underline{H_{w}}$ admits the basis $\{v_f \underline{H_w} \vert e \neq w \in W \}$.
 It follows that $R$ is based. 
\end{proof}

\begin{cor}\label{cor:Ubased}
Let $m_1, m_2, \dots, m_k \le n$. The quotient map $\pi: \VV^{\otimes (m_1+ m_2 + \cdots + m_k)} \rightarrow \wedge^{m_1} \VV \otimes \wedge^{m_2} \VV \otimes \cdots \wedge^{m_k}\VV$ is a morphism of based $\U$-modules. 
\end{cor}

\begin{proof}
The corollary follows from Lemma~\ref{lem:tensor} and Lemma \ref{lem:wedge}. 
\end{proof}

\begin{lem}
\label{lem:emb}
Let $L(\lambda_i)$ be finite-dimensional simple $\U$-modules for $i = 1$, $\dots$, $l$. We have the following based embedding of $\U$-modules for suitable $m_1$, $\dots$, $m_k \le n$: 
\[
L(\lambda_1) \otimes \cdots \otimes L(\lambda_l) \longrightarrow \wedge^{m_1} \VV \otimes \wedge^{m_2} \VV \otimes \cdots \wedge^{m_k}\VV.
\]
\end{lem}

\begin{proof}
Thanks to Lemma~\ref{lem:tensor}, it suffices to prove the statement for $l=1$.
We write $\lambda_1 = \sum^{n-1}_{i=1}a_i \omega_i \in X^+$, where $\omega_i$ is the $i$th fundamental weight.  Recall we have $\wedge^{i} \VV \cong L({\omega_i})$ for $1 \le i \le n-1$.

 Therefore we have the natural embedding 
\[
L(\lambda_1) \longrightarrow \VV^{\otimes a_1} \otimes (\wedge^2 \VV)^{a_2} \otimes \dots \otimes (\wedge^{n-1} \VV)^{a_{n-1}}.
\]
The fact that this is a based embedding follows from \cite[\S27.2]{Lu94}, since $L(\lambda_1)$ is exactly the submodule $\Big(  \VV^{\otimes a_1} \otimes (\wedge^2 \VV)^{a_2} \otimes \dots \otimes (\wedge^{n-1} \VV)^{a_{n-1}} \Big) [\ge \lambda_1]$.
\end{proof}




\section{Positivity of $\imath$-canonical bases}

\subsection{}

Let us recall the theory of canonical bases arising from quantum symmetric pairs. We first recall the following two families of quantum symmetric pairs considered in \cite{BW13}.
\begin{definition}
	\begin{enumerate}
		\item  If $n+1 = 2r+2$ is even, then we consider the quantum symmetric pair $(\Ui, \U)$ whose Satake diagram is of the form 
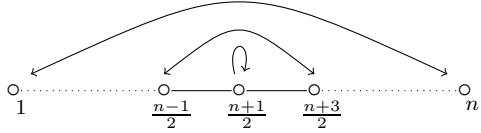
\begin{center}
\begin{tikzpicture}

 \draw[dotted]  (0 .1,0) node[left,below] {$\scriptstyle 1$} -- (1.9,0) ;
 \draw (2.1,0) node[below]{$\scriptstyle \frac{n-1}{2}$} -- (2.9,0) ;
 \draw (3.1,0) node[below]{$\scriptstyle \frac{n+1}{2}$} -- (3.9,0);
 \draw[dotted] (4.1,0) node[below] {$\scriptstyle \frac{n+3}{2}$} -- (5.9,0);
 \draw (6.1,0) node[below] {$\scriptstyle n$};
\draw (0,0) node (-r) {$\circ$};
 \draw (2,0) node (-1) {$\circ$};
\draw (3,0) node (0) {$\circ$};
\draw (4,0) node (1) {$\circ$}; 
\draw (6,0) node (r) {$\circ$};
\draw[<->] (-r.north east) .. controls (3,1.5) .. (r.north west) ;
\draw[<->] (-1.north) .. controls (3,1) ..  (1.north) ;
\draw[<->] (0) edge[<->, loop above] (0);
\end{tikzpicture}
\end{center}

The coideal subalgebra $\Ui$ of $\U$ is defined to be the $\Qq$-subalgebra of $\U$ generated by
\[
B_i = E_i +  q^{\delta_{i, \frac{n+1}{2}}}F_{n+1-i} K^{-1}_i + \delta_{i, \frac{n+1}{2}} K^{-1}_i, \quad K_i K^{-1}_{n+1-i}, \quad i = 1, 2, \dots,n.
\]
	\item  If  $n+1 = 2r+1$ is odd, then we consider the quantum symmetric pair $(\Ui, \U)$ whose Satake diagram is of the form 

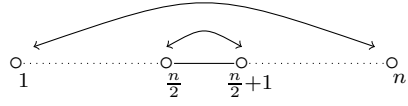
\begin{center}
\begin{tikzpicture}
 \draw[dotted]  (0.6,0) node[below]  {$\scriptstyle 1$} -- (2.4,0)  ;
 \draw (2.6,0) node[below]  {$\scriptstyle \frac{n}{2}$}
 -- (3.4,0) ;
 \draw[dotted] (3.6,0) node[below]  {$\scriptstyle \frac{n}{2}+1$} -- (5.4,0);
 \draw (5.6,0) node[below] {$\scriptstyle n$} ;
\draw (0.5,0) node (-r) {$\circ$};
 \draw (2.5,0) node (-1) {$\circ$};
\draw (3.5,0) node (1) {$\circ$}; 
\draw (5.5,0) node (r) {$\circ$};
\draw[<->] (-r.north east) .. controls (3,1) ..(r.north west) ;
\draw[<->] (-1.north) .. controls (3,0.5) ..  (1.north) ;
\end{tikzpicture}
\end{center}

	The coideal subalgebra $\Ui$ of $\U$ is defined to be the $\Qq$-subalgebra of $\U$ generated by
\[
B_i = E_i +  q^{-\delta_{i,\frac{n}{2}+1}}F_{n+1-i} K^{-1}_i, \quad K_i K^{-1}_{n+1-i}, \quad i = 1, 2, \dots,n.
\]

\end{enumerate}
\end{definition}
We shall treat both cases simultaneously. The algebra $\Ui$ admits an anti-linear involution $\psi_\imath$ such that 
\[
\psi_\imath(B_i) = B_i, \quad \psi_\imath(K_i K^{-1}_{n+1-i}) = K^{-1}_i K_{n+1-i}.
\]
There also exists  a unique (up to a scalar) intertwiner $\Upsilon$ in the completion of $\U$, such that (as an identity in the completion)
\[
 \psi_\imath (u) \Upsilon = \Upsilon \psi(u), \quad \text{for } u \in \Ui.
\]
The element $\Upsilon$ becomes  a well-defined operator for any finite-dimensional $\U$-module. For any based $\U$-module $(M, B)$ with the associated involution $\psi$, we let 
\begin{equation}\label{eq:bar}
\psi_\imath = \Upsilon \circ \psi : M \longrightarrow M.
\end{equation}

\begin{prop}\cite[Theorems~4.20 and 6.22]{BW13}\label{prop:iCB}
 Let $(M, B)$ be a based $\U$-module. Then there exists a unique $\imath$-canonical basis $B^\imath = \{b^\imath \vert b \in B\}$ of the $\Ui$-module $M$, such that 
	\begin{equation}\label{eq:iCB}
		\psi_\imath( b^\imath) = b ^\imath, \quad \text{ and } \quad b^\imath = b + \sum_{b \in B} t_{b;b'} b', \quad \text{with } t_{b;b'} \in q\Z[q].
	\end{equation}
\end{prop}

\begin{lem}\label{lem:Uibased}Let $f : M_1 \rightarrow M_2$ be a based morphism of based $\U$-modules $(M_1, B_1)$ and $(M_2, B_2)$. Then $f(B^\imath_1) \subset B^\imath_2 \cup \{0\}$.
\end{lem}

\begin{proof}

Recall $\Upsilon$ is in certain completion of $\U$, and acts (well-definedly) on any finite-dimensional  $\U$-module.  Hence $f$ commutes with the map $\psi_{\imath} = \Upsilon \circ \psi$ on $M_1$ and $M_2$. The rest of the argument (involving the partial orders and the integral lattices) is entirely similar to that of Lemma~\ref{lem:tensor}.
\end{proof}


\subsection{}
Let $W_{B_m}$ be the Weyl group of type $B_m$ with simple reflections $\{s_0, s_1, \dots, s_{m-1}\}$, where we have 
\[
\begin{split}
s^2_i = 1,  \quad \text{ for all } i,
 \\
s_i s_{i+1} s_i = s_{i+1} s_i s_{i+1}, \quad &\text{ and } \quad
s_i s_j = s_j s_i,   \qquad \text{for } |i-j| >1, 
\\
s_0 s_1 s_0 s_1=s_1s_0 s_1s_0, \quad &\text{ and } \quad s_0 s_i = s_i s_0 , \qquad \text{for } i >1.
\end{split}
\]
Let $\mc{H}_{B_m}$ be the associated Hecke algebra of type $B_m$.

Both the Hecke algebra $\mc{H}_{B_m}$ and the coideal subalgebra $\Ui$ act naturally on the tensor space $\VV^{\otimes m}$. 
\begin{prop}\cite[\S5 and \S6]{BW13}\label{prop:BD}
\begin{enumerate}
	\item The actions of $\Ui$ and $\mc{H}_{B_m}$ on $\VV^{\otimes m}$ commute. 
	\item The $\imath$-canonical basis on $\VV^{\otimes m}$ can be identified with the (parabolic) Kazhdan-Lusztig basis of $\mc{H}_{B_m}$.
\end{enumerate}
\end{prop}

\begin{prop}\label{prop:VV}
Let $\alpha, \beta \ge 0$ with $\alpha+\beta =m$. Let  $(\VV^{\otimes \alpha}, B_\alpha)$, $(\VV^{\otimes \beta}, B_\beta)$ and $(\VV^{\otimes m}, B)$ be based $\U$-modules of the tensor product of natural representations. 
	For any $b \in B$, we have 
\[
		b^\imath = \sum_{(b_\alpha, b_\beta) \in B_\alpha \times B_\beta, } t_{b;b_\alpha, b_\beta} (b^\imath_\alpha \otimes b_\beta), \qquad \text{ with }  t_{b;b_\alpha, b_\beta} \in \N[q].
\]	

\end{prop}

\begin{proof}

We consider the parabolic subalgebra $\mc{H}_{B_{\alpha}} \times \mc{H}_{A_{\beta-1}}$ of $\mc{H}_{B_m}$ generated by 
\[\{H_{s_0}, \dots, H_{s_{\alpha-1}} , H_{s_{\alpha+1}}, \dots, H_{s_{m-1}}\}.\]

It acts naturally on $\VV^{\otimes \alpha} \otimes \VV^{\otimes \beta} = \VV^{\otimes m}$. Then by Proposition~\ref{prop:A} and Proposition~\ref{prop:BD}, we see that the basis $\{b^\imath_{\alpha} \otimes b_\beta \vert b^\imath_\alpha \in B_{\alpha}, b_\beta \in B_{\beta}\}$ can be identified with the (parabolic) Kazhdan-Lusztig basis of  $\mc{H}_{B_{\alpha}} \times \mc{H}_{A_{\beta-1}}$. Therefore the proposition follows from Proposition~\ref{prop:MJ} and Proposition~\ref{prop:GH}.
\end{proof}

\begin{thm}\label{thm:BW13}
Let $L(\lambda_i)$ be finite-dimensional simple $\U$-modules for $i = 1$, $\dots$, $k$. For any $0 \le l \le k$, we consider the based modules $(L(\lambda_1) \otimes \cdots \otimes L(\lambda_l), B_{\alpha})$, $(L(\lambda_{l+1}) \otimes \cdots \otimes L(\lambda_k), B_{\beta})$, and $(L(\lambda_1) \otimes \cdots \otimes L(\lambda_k), B)$. For any $b \in B$, we have 
\[
		b^\imath = \sum_{(b_\alpha, b_\beta) \in B_\alpha \times B_\beta, } t_{b;b_\alpha, b_\beta} (b^\imath_\alpha \otimes b_\beta), \qquad \text{ with } t_{b;b_\alpha, b_\beta} \in \N[q].
\]	
\end{thm}

\begin{proof}
We first consider the case where $\lambda_i = \omega_{m_i}$ for $1 \le m_i \le n$. Recall that we have 
\[ L(\omega_{m_i}) \cong \wedge^{m_i} \VV.
\]
Let $m_1 + m_2 + \dots + m_l =\alpha$ and $m_{l+1} + m_{l+2} + \dots + m_k= \beta$ . Recall the based surjective morphism
\[
\pi = \pi_{\alpha} \otimes \pi_\beta: 
\VV^{\otimes \alpha+\beta}  \cong \VV^{\otimes \alpha} \otimes \VV^{\otimes \beta}  \twoheadrightarrow \wedge^{m_1} \VV \otimes \wedge^{m_2} \VV \otimes \cdots \otimes \wedge^{m_k}\VV .
\]

Thanks to Lemma~\ref{lem:Uibased}, all three morphisms $\pi$, $\pi_{\alpha}$, and $\pi_{\beta}$ preserve both canonical bases and $\imath$-canonical bases. Therefore applying  $\pi$ to the identity in Proposition~\ref{prop:VV}, we have 
\[
	\pi(b^\imath) = \sum_{(b_\alpha, b_\beta) \in B_\alpha \times B_\beta, } t_{b;b_\alpha, b_\beta} \Big(\pi_\alpha(b^\imath_\alpha) \otimes \pi_\beta(b_\beta) \Big), \qquad \text{ with } t_{b;b_\alpha, b_\beta} \in \N[q].
\]
The proposition follows in this case.

For  general $\lambda_i \in X^+$, the theorem follows from the previous case and the following based embedding thanks to Lemma~\ref{lem:emb} (for some suitable $m_1, m_2 , \dots, m_s $):
\[
L(\lambda_1) \otimes \cdots \otimes L(\lambda_k) \longhookrightarrow \wedge^{m_1} \VV \otimes \wedge^{m_2} \VV \otimes \cdots \otimes \wedge^{m_s}\VV.
\]
\end{proof}



\subsection{} In this section, we consider the quantum symmetric pair with another set of parameters. Their connection with the BGG category $\mc{O}$ of type $D$ was studied in \cite{ES13}. We developed the relevant theory of $\imath$-canonical bases in \cite{Bao17}, which were then used to reformulate the Kazhdan-Lusztig theory of type $D$. 

\begin{definition}
\begin{enumerate}
	\item If $n+1 = 2r+2$ is even, then the coideal subalgebra $\Ui_1$ of $\U$ is defined to be the $\Qq$-subalgebra generated by
\[
B_i = E_i + q^{\delta_{i, \frac{n+1}{2}}} F_{n+1-i} K^{-1}_i, \quad K_i K^{-1}_{n+1-i}, \quad i = 1, 2, \dots,n.
\]
	\item If  $n+1 = 2r+1$ is odd, then the coideal subalgebra $\Ui_1$ of $\U$ is defined to be the $\Qq$-subalgebra generated by
\[
B_i = E_i +  q^{-\delta_{i,\frac{n}{2}}}F_{n+1-i} K^{-1}_i, \quad K_i K^{-1}_{n+1-i}, \quad i = 1, 2, \dots,n.
\]
\end{enumerate}
\end{definition}

We shall again treat both cases simultaneously. The algebra $\Ui_1$ admits an anti-linear involution $\psi_\imath$ such that 
\[
\psi_\imath(B_i) = B_i, \quad \psi_\imath(K_i K^{-1}_{n+1-i}) = K^{-1}_i K_{n+1-i}.
\]
For any based $\U$-module $(M, B)$, we can again define the anti-linear involution (entirely similar to \eqref{eq:bar})
\[\psi_\imath : M \longrightarrow M.
\]

\begin{prop}\cite[Theorem~2.15]{Bao17}
 Let $(M, B)$ be a based $\U$-module. Then there exists a unique $\imath$-canonical basis $B^\imath = \{b^\imath \vert b \in B\}$ of the $\Ui_1$-module $M$, such that 
	\begin{equation}\label{eq:iCB}
		\psi_\imath( b^\imath) = b ^\imath, \quad \text{ and } \quad b^\imath = b + \sum_{b \in B} t_{b;b'} b', \quad \text{with } t_{b;b'} \in q\Z[q].
	\end{equation}
\end{prop}

\begin{rem}
Note that the basis we obtained here is generally different from that of Proposition~\ref{prop:iCB}, since we have taken a different subalgebra $\Ui_1$. 
\end{rem}

%
%
%

\begin{thm}\label{thm:Bao17}
Let $L(\lambda_i)$ be finite-dimensional simple $\U$-modules for $i = 1$, $\dots$, $k$. For any $0 \le l \le k$, we consider the based modules $(L(\lambda_1) \otimes \cdots \otimes L(\lambda_l), B_{\alpha})$, $(L(\lambda_{l+1}) \otimes \cdots \otimes L(\lambda_k), B_{\beta})$, and $(L(\lambda_1) \otimes \cdots \otimes L(\lambda_k), B)$. For any $b \in B$, we have 
\[
		b^\imath = \sum_{(b_\alpha, b_\beta) \in B_\alpha \times B_\beta, } t_{b;b_\alpha, b_\beta} (b^\imath_\alpha \otimes b_\beta), \qquad \text{ with } t_{b;b_\alpha, b_\beta} \in \N[q].
\]	
\end{thm}

\begin{proof}
The proof is entirely similar to that in Theorem~\ref{thm:BW13}. In this case, we can either consider the duality between the coideal subalgebra $\Ui_1$ with the Hecke algebra of type $D_m$, or with the Hecke algebra of type $B_m$ of unequal parameters. More details of the dualities can be found in \cite{Bao17, ES13}.
\end{proof}






\end{document}